\documentclass{article}

\usepackage[leqno]{amsmath}
\usepackage{amssymb,amsfonts,xfrac,MnSymbol}
\usepackage{amsthm}
\usepackage{cite}
\usepackage{hyperref}
\usepackage{todonotes}
\usepackage{enumitem}
\usepackage{graphicx}
\usepackage{tikz-cd, tikz}
\usepackage{color}
\usepackage{import}
\usepackage[capitalize]{cleveref}
\usepackage{setspace}

\numberwithin{equation}{section}

\newtheorem{theorem}{Theorem}[section]
\newtheorem{claim}[theorem]{Claim}
\newtheorem{proposition}[theorem]{Proposition}
\newtheorem{lemma}[theorem]{Lemma}

\newtheorem{corollary}[theorem]{Corollary}

\newtheorem*{theorem*}{Theorem}
\newtheorem*{claim*}{Claim}
\newtheorem*{proposition*}{Proposition}
\newtheorem*{lemma*}{Lemma}
\newtheorem*{corollary*}{Corollary}

\theoremstyle{definition}
\newtheorem{definition}[theorem]{Definition}
\newtheorem{observation}[theorem]{Observation}
\newtheorem{remark}[theorem]{Remark}

\newtheorem{fact}[theorem]{Fact}
\newtheorem{notation}[theorem]{Notation}
\newtheorem{setup}[theorem]{Setup}

\newtheorem*{definition*}{Definition}
\newtheorem*{observation*}{Observation}
\newtheorem*{remark*}{Remark}
\newtheorem*{example*}{Example}
\newtheorem*{question*}{Question}
\newtheorem*{exercise*}{Exercise}
\newtheorem*{fact*}{Fact}
\newtheorem*{notation*}{Notation}

\newcommand{\bbN}{\mathbb{N}}

\newcommand{\bbZ}{\mathbb{Z}}

\newcommand{\frC}{\mathfrak{C}}

\newcommand{\calC}{\mathcal{C}}

\newcommand{\calL}{\mathcal{L}}

\newcommand{\calN}{\mathcal{N}}

\newcommand{\actson}{\curvearrowright}

\newcommand{\into}{\hookrightarrow}
\newcommand{\inj}{\hookrightarrow}

\newcommand{\ii}{^{-1}}

\newcommand{\tild}[1]{\widetilde{#1}}

\newcommand{\rquot}[2]{{\raisebox{-.2em}{$#1$}\left\backslash\raisebox{.2em}{$#2$}\right.}}

\DeclareMathOperator{\Stab}{Stab}

\DeclareMathOperator{\Aut}{Aut}

\DeclareMathOperator{\Cay}{Cay}

\DeclareMathOperator{\rk}{rk}

\title{Ascending Chains of Free Quasiconvex Subgroups}
\author{Jack Kohav and Nir Lazarovich\thanks{Supported by the Israeli Science Foundation (grant no. 1576/23). 
}}

\begin{document}
 \onehalfspacing
\maketitle
\begin{abstract}
    We prove that a hyperbolic group cannot contain a strictly ascending chain of free quasiconvex subgroups of constant rank. 
\end{abstract}

\section{Introduction}\label{sec: intro}



We say that a group satisfies \emph{ACCF} if every ascending chain of free subgroups of constant rank stabilizes.
Takahasi \cite{takahasi1951note} and Higman \cite{higman1951finitely} showed that free groups satisfy ACCF. Kapovich-Myasnikov \cite{kapovich2002stallings} gave a new proof of this result using Stallings' foldings. Shusterman \cite{shusterman2017ascending} proved that limit groups (and in particular surface groups) satisfy ACCF. Bering-Lazarovich \cite{bering2021ascending} proved that some 3-manifold groups satisfy ACCF. Kappvich \cite{kapovich2023ascending} gave a result with similar flavours in the setting of random groups.
 
In this paper we consider the ACCF property in Gromov hyperbolic groups \cite{gromov1987hyperbolic}. 
Note that hyperbolic groups do not always satisfy ACCF; Consider a hyperbolic ascending HNN extension of the free group, e.g. the group $G=\langle a,b,t|tat^{-1}=ab, tbt^{-1}=ba\rangle$. $H=\langle a,b\rangle\leq G$ is a free subgroup of rank 2, and $H<t^{-1}Ht<t^{-2}Ht^2<...$ is an ascending chain of free subgroups of rank 2. 

To remedy this, we consider only chains of quasiconvex free subgroups, and prove the following:

\begin{theorem}\label{thm: ascending quasiconvex stabilizes}
    Let $G$ be a hyperbolic group, then every ascending chain $\{H_i\}$ of free quasiconvex subgroups of constant rank stabilizes.
\end{theorem}

Let $G$ be a hyperbolic group, and let $X$ be the Cayley graph of $G$ with respect to some fixed finite generating set. Recall that $H$ is a quasiconvex subgroup of $G$ if and only if $H$ is finitely generated and the inclusion map $H\to G$ is a quasi-isometric embedding. When $H$ is free this can be quantified as follows:
\begin{definition}\label{def: arboreal}
    Let $K\ge 1,C\ge 0$. A finitely generated free subgroup $H$ of $G$ is \emph{$(K,C)$-arboreal}, if there exists a based metric tree $(T,\tild*)$ with a free $H$-action, and an $H$-equivariant $(K,C)$-quasi-isometric embedding $(T,\tild*)\to (G,1)$.
\end{definition}
Clearly, a free subgroup of a hyperbolic group is quasiconvex if and only if it is $(K,C)$-arboreal for some $K,C$.
Our main tool in the proof of \cref{thm: ascending quasiconvex stabilizes} is to establish uniform bounds for $(K,C)$ for free quasiconvex subgroups of the same rank.

\begin{theorem} \label{thm: quasiconvex is arboreal}
    Let $G$ be a hyperbolic group.
    For all $r\in \bbN$ there exist $K\ge 1,C\ge 0$ such that every quasiconvex free subgroup of $G$ of rank $r$ is $(K,C)$-arboreal.
\end{theorem}

The proof of this theorem is through a process of Stallings' folds for quasiconvex subgroups of hyperbolic groups. See also \cite{stallings1983topology,weidmann2024foldings,arzhantseva1996class,kapovich2004freely,bestvina1991bounding,dunwoody1998folding,kapovich2005foldings,weidmann2002nielsen}.
We show that if $H\le H'$ are $(K,C)$-arboreal, then there is $(K',C')$-quasi-isometric map between the core of $H$ to the core of $H'$. If in addition $H$ is not contained in a free factor of $H'$ we use this map to bound the size of the core of $H'$:

\begin{theorem}\label{thm: main_plus}
Let $G$ be a hyperbolic group, let $r\in \bbN$ and let $H\le G$ be a quasiconvex free subgroup. Then there are finitely many free quasiconvex subgroups $H\le H'\le G$ of rank $r$ in which $H$ is not contained in a free factor.
\end{theorem}




\cref{thm: ascending quasiconvex stabilizes} easily follows from \cref{thm: main_plus} as we explain in \S\ref{sec: ascending chains}.

\paragraph*{Acknowledgements.}
We would like to thank Edgar Bering for many fruitful conversations. 

\section{Outline}\label{sec: outline}
In this section we review Kapovich and Myasnikov's  proof \cite{kapovich2002stallings} of the following theorem, and see how aspects of it might carry over to the more general case of hyperbolic groups.
\begin{theorem}[Takahasi \cite{takahasi1951note}, Higman \cite{higman1951finitely}]
    The free group $F_n$ satisfies ACCF.
\end{theorem}

\begin{proof}[Proof \cite{kapovich2002stallings}]
Let $|S|=k$, and $G=F(S)$ be the free group on $S$. Let $H_1\le H_2\le \dots \le G$ be an ascending chain of (free) subgroups of rank $k$. As we will see in \cref{lem: ascneding not in free factors} we may assume that $H_i$ is not contained in a free factor of $H_{i+1}$ for all $i$. Now, consider the action of one of the groups $H_i$ on $\Cay(G,S)$. For each one of the groups $H_i$, let $T_i$ be the convex hull of the orbit of $1\in\Cay(G,S)$. $T_i$ is an $H_i$-invariant subtree and the quotient $\Gamma_i=\rquot{H_i}{T_i}$ is a finite metric graph with $\pi_1(\Gamma_i)\simeq H_i$. The inclusion $T_i\subseteq T_{i+1}$ (which is due to the fact that $T_{i+1}$ is also $H_i$-invariant) induces a map $\varphi_i:\Gamma_i\to\Gamma_{i+1}$. These maps are surjective - otherwise, $H_i$ would have been contained in a free factor of $H_{i+1}$. Since the maps $\varphi_i$ are simplicial, the number of edges in $\Gamma_i$ must reduce. This means that the chain of graphs $\Gamma_i$ stabilizes. By the definition of $T_i$, the groups $H_i$ must eventually stabilize as well.
\end{proof}
 
The key part of the proof is finding a cocompact invariant subtree, and using it to attach a finite graph to every subgroup in the chain. In hyperbolic groups, it is not clear how to do this - the group does not have a unique minimal invariant subtree. To remedy this, we define the notion of arboreality (\cref{def: arboreal}). Arboreal subgroups are assumed to come with a finite graph attached, and a quasi-isometry sending the universal cover to the Cayley graph of our group. In sections \S\ref{sec: folding}-\S\ref{sec: uniform constants} we will show that indeed every quasiconvex free subgroup is arboreal with constants that are determined by the rank. After that, in \S\ref{sec: maps between cores}-\S\ref{sec: ascending chains} we look at nested subgroups that are arboreal with the same constants. we can work with chains of subgroups that are all arboreal with the same constants.

In \S\ref{sec: folding} we define a notion of folding (inspired by Stallings' foldings) so that we are able to improve cores that do not meet our criteria. 

In \S\ref{sec: expanding quasi isometry} we take a core that decomposes into two good subcores, and prove that if this decomposition is optimal in the sense that there are no foldings, the big core is also good enough. We use local to global - the fact that no folding is possible gives us all the local data we need.

In \S\ref{sec: uniform constants} we put the pieces together, to prove \cref{thm: quasiconvex is arboreal}. 

Once we finish with \cref{thm: quasiconvex is arboreal}, we can follow the proof above and generalize it to hyperbolic groups. 

The proof starts with the reduction that lets us assume $H_i$ is not contained in a proper free factor of $H_{i+1}$. As seen in \cref{lem: ascneding not in free factors}, this reduction is possible even in the case of hyperbolic groups. In \cref{lem: first not free factor}, we go a step further and show that $H_1$ is not contained in a proper free factor of $H_{i}$.

Arboreality attaches a graph to a subgroup. But given an inclusion $H_1<H_2$, how is the graph attached to $H_1$ related to the graph attached to $H_2$? In the original proof, there is an inclusion of the invariant subtrees which induces a map between the quotients. In the hyperbolic case, we do not have this inclusion. In \S\ref{sec: maps between cores} we use properties of hyperbolic spaces to build maps between cores. We cannot guarantee these maps are simplicial, but they are quasi isometries.

After constructing maps between cores (that are surjective and $\pi_1$-injective), the original proof uses the fact that these are simplicial maps between finite graphs to conclude the cores stabilize. For hyperbolic groups, we cannot do this. Instead, we use the fact that the maps are quasi-isometries, and that Cayley graph are locally finite. In \S\ref{sec: bounding cores}, we prove \cref{thm: main_plus}. Then, in \S\ref{sec: ascending chains} we finish everything by proving \cref{thm: ascending quasiconvex stabilizes}

\section{Preliminaries}\label{sec: prelims}
\begin{notation}
    Throughout the article (with \S\ref{sec: ascending chains} as the exception) we work with a hyperbolic group $G$. We also fix a Cayley graph which we denote by $X$. $X$ is assumed to be $\delta$-hyperbolic. The metric is denoted by $d_X$, and the $K$-neighborhood of a subset $A\subset X$ is denoted by $\calN_K(A)$. We denote the action of $G$ on $X$ by multiplication - $g\in G$ maps $x\in X$ to $g\cdot x$. For $p,q\in X$, we denote a geodesic between $p$ and $q$ by $[p,q]_X$. There is some ambiguity since geodesics are not necessarily unique, but it should be clear from the context whether we take a specific geodesic or an arbitrary one. Finally, given a graph $\Gamma$ we denote its universal cover by $\tild\Gamma$, and the covering map by $\pi_\Gamma$. 
\end{notation}
\begin{definition}\label{def: core}
    Suppose $H\leq G$ is a free quasiconvex subgroup. An \emph{$H$-core} is a triplet $\frC=(\Gamma, \rho, \iota)$ such that:
    \begin{enumerate}
        \item $\Gamma$ is a finite metric graph
        \item $\rho:H\inj\Aut(\tild\Gamma)$ 
        is a free isometric action 
        with $\rquot{H}{\tild\Gamma}=\Gamma$ 
        \item $\iota:\tild\Gamma\looparrowright X$ is an $H$-equivariant immersion that maps edges of $\tild\Gamma$ to geodesics of $X$ isometrically.
    \end{enumerate}
    A core $\frC$ is called \emph{$(K,C)$-good} if $\iota$ is a $(K,C)$-quasi isometric embedding.
    A \emph{based $H$-core} is a quadruplet $\frC_*=(\Gamma,\tild*,\rho,\iota)$ such that $(\Gamma,\rho,\iota)$ is an $H$-core and $\tild*\in\tild\Gamma$ is a vertex such that $\iota(\tild*)=1\in X$. We denote $*=\pi_\Gamma(\tild*)\in\Gamma$.
    The \emph{size of a core} $\sigma(\frC)=\sigma(\Gamma)$ is the sum of the lengths of edges in $\Gamma$.
\end{definition}
\begin{remark}
    In a based $H$-core $\frC_*=(\Gamma,\tild*,\rho,\iota)$, we get that $\rho$ is an isomorphism between $\pi_1(\Gamma,*)$ and $H$.
\end{remark}
\begin{remark}
  Note that $H$ is $(K,C)$-arboreal if and only if it admits a $(K,C)$-good based core.  
\end{remark}



\section{Folding cores}\label{sec: folding}
We move towards finding a proof of \cref{thm: quasiconvex is arboreal}. We do this using foldings. 
\begin{definition}\label{def: simple folding}
    Let $\Gamma$ be a connected metric graph.
    A \emph{simple folding of $\Gamma$ along an edge $e$} is a connected metric graph $\Gamma'$ and a homotopy equivalence $F:\Gamma \to \Gamma'$ which is one of the following:
    \begin{enumerate}
        \item Replacing an edge: $\Gamma' = (\Gamma - e) \cup e'$  (so that $\Gamma - e \cup e'$ is connected),
        \item Identifying a pair of vertices: $\Gamma'$ is obtained from $\Gamma - e$ by identifying a pair of vertices, or
        \item Adding a balloon: if $e$ is non-separating, $\Gamma'$ is obtained from $\Gamma - e$ by attaching a balloon, i.e. an edge with a loop at its end.
    \end{enumerate}
    In all cases $F$ satisfies $F|_{\Gamma -e}$ is the obvious inclusion (or quotient) map, and $e$ is mapped linearly to a new path between the same endpoints. 
\end{definition}
A simple folding is a way of modifying the graph $\Gamma$ by homotopy that changes only the edge $e$ (and glues it in a new way). We want to expand this notion to cores. A core comes with a metric and a map from the universal cover to $X$. We want the folding to relate to this extra data in the same way as before - not changing it on $\Gamma-e$, but allowing some change on $e$.
\begin{definition}\label{def: improvement}
    Suppose $\frC=(\Gamma,\rho, \iota)$ and $\frC'=(\Gamma',\rho',\iota')$ are $H$-cores, and $F:\Gamma\to\Gamma'$ is a simple folding along an edge $e\in\Gamma$. $\frC'$ is an \emph{improvement of $\Gamma$ along $e$}, if:
    \begin{enumerate}
        \item $F|_{\Gamma-e}$ is length-preserving.
        \item For a suitable lift $\tild F:\tild\Gamma\to\tild\Gamma'$ of $F$, and $x\in\tild\Gamma-\pi_\Gamma^{-1}(e)$ we have $\iota(x)=\iota'(\tild F(x))$.
        \item $\sigma(\frC')<\sigma(\frC)$
    \end{enumerate}

    The edge $e$ is \emph{minimal} if there are no improvements of $\frC$ at $e$.
    An edge $\tild e$ of $\tild\Gamma$ is \emph{minimal} if $\pi(\tild e)$ is minimal.
\end{definition} 
The space $\tild\Gamma-\pi_\Gamma^{-1}(e)$ is a union of trees that is mapped into $X$. the lifts of $\tild{e}$ are connecting the components of $\tild\Gamma-\pi_\Gamma^{-1}(e)$. To find an improvement at $e$, we can look inside $X$ and try to find other ways to connect the components of the image $\iota(\tild\Gamma-\pi_\Gamma^{-1}(e))$. Such a replacement should reduce the length of the quotient.

\section{Expanding Quasi Isometry}\label{sec: expanding quasi isometry}

Throughout this section, we work with the following: \begin{setup}\label{setup}
    $H$ is a quasiconvex free subgroup, $\frC=(\Gamma,\rho,\iota)$ is a core for $H$, and $e\in \Gamma$ is a minimal edge.
    Let $\tild e \in \pi\ii(e)$.
    Let $\tild v_1,\tild v_2$ be the two endpoints of $\tild e$, and let $\tild \Gamma_i$ be the component of $\tild \Gamma - \pi\ii(e)$ which contains $\tild v_i$.
    Let $H_i = \Stab_H(\tild \Gamma_i)$ and $\Gamma_i =  \rquot{H_i}{\tild\Gamma_i} \subset \Gamma$.
    Denote by $\frC_i$ the $H_i$-core given by $\frC_i=(\Gamma_i,\rho_i,\iota_i)$ where $\rho_i$ is the action of $H_i$ on $\tild \Gamma_i$ and $\iota_i = \iota|_{\tild \Gamma_i}$.
\end{setup}
\begin{remark}
        If $e$ separates then $\Gamma_1\ne \Gamma_2$ and $H=H_1 * H_2$.
        Otherwise, if $e$ is non-separating, $\Gamma_1 = \Gamma_2$, the subgroups $H_1,H_2$ are conjugate subgroups of $H$ and $H=H_1 * \bbZ = H_2 * \bbZ$.
    \end{remark}

The goal of this section is to prove the following:

\begin{proposition}\label{prop: constants exist}
    For all $K,C$ there exist $K',C', L$ such that if in \cref{setup}, $\ell(\tild e)\geq L$ and 
    $\frC_1,\frC_2$ are $(K,C)$-good, then $\frC$ is $(K',C')$-good.
\end{proposition}

\begin{lemma}\label{obs: minimal edge implies shortest path}
    In \cref{setup}, the image $\iota(\tild e)$ is a shortest possible path in $X$ between $\iota_1(\tild\Gamma_1)$ and $\iota_2(\tild\Gamma_2)$. 
\end{lemma}
\begin{proof}
    Suppose for contradiction that $\iota(\tild e)$ is not a shortest path between $\iota_1(\tild\Gamma_1)$ and $\iota_1(\tild\Gamma_2)$. We will contradict the minimality of $\tild e$ by constructing an improvement of $\Gamma$ at $e=\pi_\Gamma(\tild e)$. We will construct the improvement $\Gamma'$ in two ways, depending on two cases.
    
    \textbf{Case 1.} Suppose that $\iota_1(\tild\Gamma_1)$ and $\iota_2(\tild\Gamma_2)$ intersect at a vertex $g$. Then there exist two vertices $\tild{u_1}\in\tild\Gamma_1,\tild{u_2}\in\tild\Gamma_2$ such that $\iota_1(\tild{u_1})=\iota_2(\tild{u_2})=g$. 
    Let $u_i = \pi(\tild u_i)$. We claim that $u_1\ne u_2$: as otherwise the path $\tild \eta$ connecting $\tild u_1,\tild u_2$ descends to a non-trivial loop $\eta = \pi\circ \tild \eta\in \pi_1(\Gamma)\simeq H$ which satisfies $\eta.\tild u_1= \tild u_2$ and so by $H$-equivariance $\eta\cdot g=g$; however the action on $X$ is free.
    
    Let $\gamma_i$ be the path from $\tild v_i$ to $\tild u_i$ in $\tild\Gamma_i$, and let $\gamma_i = \pi \circ \tild \gamma_i$.
    Let $\Gamma'$ be the graph obtained from $\Gamma$ by removing $e$ and identifying $u_1$ with $u_2$, i.e. $ \Gamma ' = (\Gamma - e)/u_1 \sim u_2$. Let $F:\Gamma\to\Gamma'$ be the folding map which is the quotient map on $\Gamma - e$ and sends $e$ to the concatenation $\gamma_1 \bar \gamma_2$ (where $\bar \gamma_2$ is the reverse path of $\gamma_2$), and let $\tild F$ be a lift of $F$ to the universal cover. The  map $\iota$ on $\tild \Gamma - \pi\ii(e)$ defines a map $\iota':\tild{\Gamma'}\to X$.
    Since $F$ is a homotopy equivalence and $H$ is the deck transformation group of $\tild\Gamma$, we get an action $\rho'$ of $H$ on $\tild\Gamma'$. By the definition of $\iota'$, we have $\iota'(h.x)=h.\iota'(x)$ for $x\in\tild\Gamma',h\in H$. Therefore, $(\Gamma', \varphi',\iota')$ is an $H$-core.  
    Now that we know that $(\Gamma',\rho',\iota')$ is an $H$-core, we can see that it is an improvement of $\Gamma$ at $e$: In \cref{def: improvement}, properties 1,2 follow from the definition of $\iota'$, 3 holds by the definition of $\rho'$, and 4 follows from the definition of $\Gamma'$. Therefore, $e$ is not minimal.

    \textbf{Case 2.} Suppose that $\iota_1(\tild\Gamma_1)$ and $\iota_2(\tild\Gamma_2)$ do not intersect. Let $\beta$ be a shortest path in $X$ between $\iota_1(\tild\Gamma_1)$ and $\iota_2(\tild\Gamma_2)$. Let $\tild u_1\in\Gamma_1,\tild u_2\in\Gamma_2$ such that $\iota(\tild u_1)$ and $\iota(\tild u_2)$ are the endpoints of $\beta$. We denote $u_1=\pi_\Gamma(\tild u_1),u_2=\pi_\Gamma(\tild u_2)$. Let $\tild\gamma_i$ be the path from $\tild v_i$ to $\tild u_i$ in $\tild\Gamma_i$ and let $\gamma_i=\pi_\Gamma\circ\tild\gamma_i$. Let $\Gamma'=(\Gamma-e)\cup e'$, where $\ell(e')=\ell(\beta)$ and the endpoints of $e'$ are $v_1$ and $v_2$. Let $F$ be the folding map which is the identity map on $\Gamma-e$ and sends $e$ to the concatination $\gamma_1e'\bar\gamma_2$. We can define the map $\iota':\tild{\Gamma'}\to X$ in the following way: On $\tild\Gamma'-\pi_{\Gamma'}^{-1}(e')$, we use $\iota$, and lifts of $e'$ are mapped isometrically to the appropriate copies of $\beta$. Since $F$ is a homotopy equivalence we get an action $\rho'$ of $H$ on $\tild\Gamma'$. We now have a triplet $(\Gamma',\rho',\iota')$ which is an $H$-core, and again, it is an improvement of $\Gamma$ at $e$. We conclude that $e$ is not minimal.
\end{proof}


\begin{lemma}\label{lem: upperbound on gromov prod}
    For all $K,C$, there exists $M_1$, such that if in \cref{setup}, $\frC_i$ is $(K,C)$-good  $x\in \tild e, y\in \tild \Gamma_i$ then $(\iota(x)|\iota(y))_{\iota(\tild v_i)}<M_i$.
\end{lemma}

\begin{proof}
    Without loss of generality let us prove it for $i=1$.
    Let $\gamma_1$
    be the geodesic between $\tild v_1$ and $y$.
    $\iota(\gamma_1)$ is a $(K,C)$-quasi geodesic. By the Morse lemma we have that $\iota(\gamma_1)$ is $M_0$-close to $[\iota(y),\iota(\tild v_1)]_X$, for $M_0=M_0(K,C)$. Set $M_1 = M_0+2\delta+1$.
    
    Assume towards contradiction that $(\iota(x)|\iota(y))_{\iota(\tild v_1)}\ge  M_1$. 
    Then the point $p$ on $\iota(\tild e)$ at distance $M_1$ from $\iota(\tild v_1)$ is at distance $M_0+2\delta$ from some point in $\iota_1(\tild \Gamma_1)$: since it is at distance $2\delta$ from some point $p'$ on $[\iota(x),\iota(y)]$ and $p'$ is at distance $M_0$ from a point on $\iota(\gamma_1)$.
    This contradicts  \cref{obs: minimal edge implies shortest path} since $M_0 + 2\delta < M_1$.
\end{proof}

\begin{lemma}\label{lem: expanding by minimal edge}
    For all $K,C$ there exist $C'$ such that the following holds: If $\frC_i$ is $(K,C)$-good, then $\iota|_{\tild e\cup \tild\Gamma_i}$ is a $(K,C')$-quasi-isometric embedding. 
\end{lemma}
\begin{proof}
Without loss of generality let us prove it for $i=1$.
The map $\iota$ is length-preserving so it is $1$-Lipschitz. This gives one of the inequalities of the quasi-isometry, namely $d_X(\iota(x),\iota(y)) \le d_\Gamma(x,y)$.

For the other inequality, let $x,y\in \tild e\cup\tild\Gamma_1$. If both $x,y$ are in $\tild\Gamma_1$, then $d_X(\iota(x),\iota(y)) \ge K\ii d_\Gamma(x,y) -C$. Similarly, if both $x,y\in \tild e$ then $d_X(\iota(x),\iota(y))=d_\Gamma(x,y)$. So assume without loss of generality that $x\in\tild e,y\in\tild \Gamma_1$.
    

    By Lemma \cref{lem: upperbound on gromov prod}  $(\iota(x)|\iota(y))_{\iota(\tild v_1)}<M_1$ and so we get the desired lower bound as follows:
    \begin{align*}
    d_X(\iota(x),\iota(y))&=d_X(\iota(x),\iota(v)))+d_X(\iota(v),\iota(y)-2(\iota(x)|\iota(y))_{\iota(\tild v_1)}\\
    &\geq d_{\tild\Gamma}(x,v)+\frac{1}{K}d_{\tild\Gamma}(v,y)-C-2M_1\\
    &\geq \frac{1}{K}(d_{\tild\Gamma}(x,v)+d_{\tild\Gamma}(v,y))-C-2M_1\\
    &=\frac{1}{K}d_{\tild\Gamma}(x,y)-C-2M_1
    \end{align*}
    Setting $C' = C+2M_1$ completes the proof.
\end{proof}

\begin{lemma}\label{lem: expanding by two copies of minimal edge}
    For all $K,C$ there exist $C''$ such that if in \cref{setup}, $\frC_i$ is $(K,C)$-good, and $\tild e'$ is another lift of $e$ with an endpoint in $\tild \Gamma_i$ then $\iota|_{\tild e\cup\tild e'\cup \tild \Gamma_i}$ is a $(K,C'')$-quasi isometric embedding.
\end{lemma}
\begin{proof}
    Without loss of generality let us prove it for $i=1$. 
    Let $x,y\in\tild e\cup\tild e'\cup \tild\Gamma_1$. If one of $x,y$ is in $\tild\Gamma_1$, then by \cref{lem: expanding by minimal edge} we are done (if we choose constants which are at least those in that lemma). Assume without loss of generality that $x\in \tild e, y\in\tild e'$.
    As in the proof of \cref{lem: expanding by minimal edge}, the inequality $d_X(\iota(x),\iota(y)) \le d_{\tild \Gamma}(x,y)$ is clear, so our goal is to bound $d_X(\iota(x),\iota(y))$ from below.
    
    To simplify notation let $v=\tild v_1$ be the vertex of $\tild e$ that lies in $\tild\Gamma_1$. Similarly, let $v'$ be the vertex of $\tild e'$ that lies in $\tild \Gamma_1$. 

    \begin{claim}\label{claim: second bound on gromov prod}
        There exists  a constant $M_2 = M_2(K,C)$ such that $(\iota(x)|\iota(y))_{\iota(v)}<M_2$ or $(\iota(x)|\iota(y))_{\iota(v')}<M_2$.
    \end{claim}
    
    Assuming the claim, the proof is completed as follows: Without loss of generality, assume that $(\iota(x)|\iota(y))_{\iota(v)}<M_2$. By \cref{lem: expanding by minimal edge}, there exist a constant $C'$ (that depend only on $K,C$) for which $$d_X(\iota(v),\iota(y))\ge \frac{1}{K}d_{\tild \Gamma} (v,y) - C'$$. It now follows that:
    \begin{align*}    d_X(\iota(x),\iota(y))&=d_X(\iota(x),\iota(v))+d_X(\iota(v),\iota(y))-2(\iota(x)|\iota(y))_{\iota(v)}\\
    &\geq d_{\tild\Gamma}(x,v)+\frac{1}{K}d_{\tild\Gamma}(v,y)-C'-2M_2\\
    &\geq \frac{1}{K}(d_{\tild\Gamma}(x, v)+d_{\tild\Gamma}(v,y))-C'-2M_2\\
    &=\frac{1}{K}d_{\tild\Gamma}(x,y)-C'-2M_2
    \end{align*}
\end{proof}
    \begin{proof}[Proof of \cref{claim: second bound on gromov prod}]
    There is an element $h\in H$ that sends $\tild e$ to $\tild e'$. If $h.v=v'$, then $h\in H_1=\Stab(\tild \Gamma_1)$. Otherwise, $e$ is a non-seperating edge.
    We now divide the proof to cases:
    
    \textbf{Case 1 - $d_X(\iota(v),\iota(v'))> 2M_1+\delta$.}
    We have 
    \begin{equation*}
    2M_1+\delta<d_X(\iota(v),\iota(v'))=(\iota(y)|\iota(v'))_{\iota(v)}+(\iota(y)|\iota(v))_{\iota(v')}
    \end{equation*}
    By \cref{lem: upperbound on gromov prod}, $(\iota(y)|\iota(v))_{\iota(v')}<M_1$ and so 
    \begin{equation}\label{eq: Case 1. eq1}
        (\iota(y)|\iota(v'))_{\iota(v)}>M_1+\delta.
    \end{equation}
    
    By hyperbolicity:    
    \begin{equation*}\label{eq:hyperbolicity in case 3.1}
    (\iota(x)|\iota(v'))_{\iota(v)}\ge\min\{(\iota(x)|\iota(y))_{\iota(v)},(\iota(y)|\iota(v'))_{v)}\}-\delta.
    \end{equation*}
    By \cref{lem: upperbound on gromov prod}, $(\iota(x)|\iota(v'))_{\iota(v)}<M_1$, so 
    \begin{equation}\label{eq: Case 1. eq2}
        \min\{(\iota(x)|\iota(y))_{\iota(v)},(\iota(y)|\iota(v'))_{\iota(v)}\}<M_1+\delta.
    \end{equation}
    
    By \eqref{eq: Case 1. eq1} and \eqref{eq: Case 1. eq2}, we conclude that 
    \begin{equation*}\label{eq: case 1 bound on gromov product}
        (\iota(x)|\iota(y))_{\iota(v)}<M_1+\delta.
    \end{equation*}
    This finishes case 1, and as explained before, this gives: 
    \begin{equation}\label{eq: case 1 conclusion}
    d_X(\iota(x),\iota(y))\geq \frac{1}{K}d_{\tild\Gamma}(x,y)-C'-2M_1-2\delta
    \end{equation}
    
    \textbf{Case  2 - $d_X(\iota(v),\iota(v'))\leq 2M_1+\delta$, and $h.v\neq v'$.}

    Denote the path between $v$ and $v'$ by $\gamma$.    
    Recall that by the Morse Lemma, $[\iota(v),\iota(v')]_X$ is $M_0$-close to $\iota(\gamma)$ for $M_0=M_0(K,C)$.
    Suppose for contradiction that 
    \begin{equation*}
        (\iota(x)|\iota(y))_{\iota(v)}>2M_0+4\delta+3\quad \text{and} \quad          (\iota(x)|\iota(y))_{\iota(v')}>2M_0+4\delta+3. 
    \end{equation*}
    Let $p\in [\iota(x),\iota(v)]_X\subseteq \iota(\tild e)$ and $q\in[\iota(y),\iota(v')]_X\subset\iota(\tild e')$ be such that 
    \begin{equation}\label{eq: Case 2 - eq 1}
    d_X(p,\iota(v))=d_X(q,\iota(v'))=M_0+4\delta+1
    \end{equation}
    We claim that $d_X(p,q)\leq 4\delta$.  
    Since $d_X(p,\iota(v))<(\iota(x)|\iota(y))_{\iota(v)}$, $p$ is $\delta$-close to some point $p'\in [\iota(y),\iota(v)]_X$.  
    By \cref{obs: minimal edge implies shortest path}, $[\iota(p),\iota(v)]_X\subseteq \iota(\tild e)$ is the shortest path between $\iota(p)$ and $\iota(\tild\Gamma_1)$, and therefore
    $$d_X(p',\iota(\tild\Gamma_1))\geq d_X(p,\iota(v))-\delta>M_0+3\delta.$$
    Since $[\iota(v),\iota(v')]$ is contained in the $M_0$-neighborhood of $\iota_1(\tild\Gamma_1)$, 
    $p'$ cannot be $\delta$-close to a point in $[\iota(v),\iota(v')]$. The triangle with vertices $\iota(y),\iota(v),\iota(v')$ is $\delta$-thin, so $p'$ is $\delta$-close to some $p''\in [\iota(y),\iota(v')]_X$. Hence, 
    \begin{equation}\label{eq: Case 2 - eq 2}
        d(p,p'')\le 2\delta.
    \end{equation}
    
    By \cref{obs: minimal edge implies shortest path}, $[\iota(y),\iota(v')]$ is the shortest path between $\iota(y)$ and $\iota(\tild\Gamma_1)$, and therefore $$d_X(p'',\iota(v'))\geq d_X(p'',\iota(v))\geq d_X(p,\iota(v))-2\delta.$$
    Similarly, since $[\iota(p),\iota(v)]$ is the shortest path between $p$ and $\tild\Gamma_1$,
    $$d_X(p,\iota(v))\geq d_X(p,\iota(v'))\geq d_X(p'',\iota(v'))-2\delta.$$ 
    The last two inequalities and \eqref{eq: Case 2 - eq 1} give
    $$d_X(p'',q)=|d_X(p'',\iota(v'))-d_X(q,\iota(v'))|<2\delta$$
    And by \eqref{eq: Case 2 - eq 2}
    \begin{equation}\label{eq: Case 2 eq 3}
        d_X(p,q)\leq d_X(p,p'')+d_X(p'',q)\leq 4\delta.
    \end{equation}
    We have that $\ell(e)\geq2M_0+4\delta+3>2d(p,\iota(v))$. This means that $$d_X(h\cdot p,\iota(v'))>\frac{1}{2}\ell(e)>d_X(q,\iota(v')).$$
    Since $q,h\cdot p,\iota(v')$ lie in the same geodesic, we get
    \begin{equation}\label{eq: Case 2 eq 4}
        d_X(\iota(v'),h\cdot p)= d_X(\iota(v'),q)+d_X(q,h\cdot p)
    \end{equation}
    Combining \eqref{eq: Case 2 - eq 1},\eqref{eq: Case 2 eq 3}, and \eqref{eq: Case 2 eq 4} we get
    \begin{align*}
        d_X(p,h\cdot p) &\le d_X(p,q) + d_X(q,h\cdot p) \\ 
        &\le d_X(p,q)- d_X(\iota(v'),q) + d_X(\iota(v'),h\cdot p) \\
        &\le 4\delta - (M_0+4\delta+1) + d_X(\iota(v'),h\cdot p) \\
        &< d_X(\iota(v'),h\cdot p).
    \end{align*}
    We have now discovered that the point $h.p$ is closer to its translate $p$ then it is to the vertex $\iota(v')$. As we will show next, this means we can improve $\frC$ by replacing the non-separating edge $e$ with a balloon. 

    Let $\Gamma'=(\Gamma-e)\bigcup({e_1\cup u\cup e_2})$ where $u$ is a vertex, $e_1$ connects $\pi_\Gamma(v)$ to $u$, and $e_2$ is a loop at $u$. Set $\ell(e_1)=d_X(p,\iota(v))$ and $\ell(e_2)=d_X(p,h\cdot p)$. Define a quotient map $F:\Gamma\to\Gamma'$ by subdividing $e$ into four segments, mapping the first to $e_1$, mapping the second segment to $e_2$, the third to $\bar e_1$ (i.e. $e_1$ with its opposite orientation) and mapping the fourth segment to $\pi_\Gamma \circ \gamma$. $F$ is a homotopy equivalence -- a homotopy inverse can be constructed by sending $e_1$ to $v$ and $e_2$ to the concatenation of $e$ and $\gamma$. Now, we want an $H$-core structure for $\Gamma'$. Since $F$ is a homotopy equivalence, the action $H\actson_\rho \tild\Gamma$ induces an action of $H\actson_{\rho'}\tild{\Gamma'}$ once we fix a lift $\tild F$ of $F$. We note that there is a bijection between lifts of $\Gamma_1\subset\Gamma$ to $\tild\Gamma$ and lifts of $F(\Gamma_1)\subset\Gamma'$ to $\tild\Gamma'$. Now to define $\iota'$, we want to map lifts of $\Gamma_1$ in the same way that $\iota$ maps lifts of $\Gamma_1$. The remaining parts of $\tild{\Gamma'}$ are lifts of the balloon. We want lifts of $u$ to be mapped to points in the orbit of $p$, and lifts of $e_1$ and $e_2$ to be mapped isometrically to the relevant geodesics. This gives us a well defined $H$-equivariant map $\iota':\tild{\Gamma'}\to X$, and thus $\frC'=(\Gamma',\rho',\iota')$ is an $H$-core.
    We now want to show it is an improvement of $\frC$. In \cref{def: improvement}, property 1 follows from the definition of $F$, properties 2 and 3 follow from the definitions of $\iota',\rho'$, and property 4 follows from the fact that $$\ell(e_1)+\ell(e_2)=d_X(p,\iota(v))+d_X(p,h\cdot p)<d_X(p,\iota(v))+d_X(h\cdot p, \iota(v'))=\ell(e).$$
    We have found an improvement of $\frC$ at $e$, which is a contradiction to minimality. Therefore, we conclude that $(\iota(x)|\iota(y))_{\iota(v)}\leq M_0+3\delta$ or $(\iota(x)|\iota(y))_{\iota(v')}\leq M_0+3\delta$.
    
    \textbf{Case 3 - $d_X(\iota(v),\iota(v'))\leq 2M_1+\delta$, and $h.v=v'$.}
    
    In this case, $h\in H_1=\Stab(\tild\Gamma_1)$. We would like to use case 1 by looking at a translation of $\tild{e}$ by a power of $h$.

    Recall that we want to bound $D:=(\iota(x)|\iota(y))_{\iota(v)}$. We may assume that $D>M_0+3\delta$, as otherwise we are done. Let $p\in [\iota(x),\iota(v)]_X,q\in[\iota(y),\iota(v')]_X$ such that $$d_X(p,\iota(v))=d_X(q,\iota(v'))=D$$
    Since $D>M_0+3\delta$, just like in the previous case, we have that $d_X(p,q)\leq 4\delta$. However, in this case, $q=h\cdot p$. Therefore, for all $m\in \bbN$,
    \begin{equation}\label{Case 3 eq upper}
        d_X(p,h^m\cdot p) \le m d_X(p,h \cdot p) \le  4\delta m.
    \end{equation}

    On the other hand, the translation length of any non-trivial element in $H_1$ in its action on $\tild \Gamma$ is at least 1, and so we have $d_{\tild\Gamma}(v,h^m.v)\ge m$. The map $\iota_1 = \iota | _{\tild \Gamma_1}$ is assumed to be a $(K,C)$-quasi-isometric embedding.
    This implies that
    \begin{equation*}
    d_X(\iota(v),h^m\cdot\iota(v))=d_X(\iota(v),\iota(h^m.v))\ge \frac{1}{K}m-C.
    \end{equation*}
    Thus, for $m=K(2M_1+\delta +C+1)$ we have
    \begin{equation*}
    d_X(\iota(v),h^m.\iota(v))> 2M_1+\delta.
    \end{equation*}
    Suppose $p=\iota(s)$ for $s$ in $\tild e$. We have
    \begin{equation*}
        d_{\tild\Gamma}(s,h^m.s)=2d_{\tild\Gamma}(s,v)+d_{\tild\Gamma}(v,h^m.v)>2d_{\tild\Gamma}(s,v)=2D
        .
    \end{equation*}
    The points $s$ and $h^m.s$ satisfy Case 1, so by \eqref{eq: case 1 conclusion},
    \begin{equation}\label{eq: Case 3 eq 1}
        d_X(p,h^m\cdot p)\geq \frac{1}{K}d_{\tild\Gamma}(s,h^m.s)-C'-2M_1-2\delta > \frac{2}{K}D-C'-2M_1 -2\delta.
    \end{equation}
    Combining \eqref{Case 3 eq upper} and \eqref{eq: Case 3 eq 1} we get 
    $$  \frac{1}{K}2D-C'-2M_1 -2\delta < 4\delta m$$
    This gives $$(\iota(x)|\iota(y))_{\iota(v)}=D < \tfrac12 K ( C'+2M_1 +2\delta +4\delta m)$$
    This upper bound depends only on $K,C$ and so we are done.
\end{proof}
\begin{lemma}\label{lem: full map is local qi}
        For all $K,C$ there exists $C'$ such that if in \cref{setup}, $\frC_1,\frC_2$ are $(K,C)$-good, $\frC$ is $(K,C')$-good, $\iota$ is an $\ell(\tild e)$-local quasi isometric embedding.
\end{lemma}
\begin{proof}
    Let $x,y\in\tild\Gamma$ such that $d_{\tild\Gamma}(x,y)\leq \ell (\tild e)$. Up to an isometry, we may assume $x\in\tild e\cup\tild\Gamma_1\cup\tild\Gamma_2$. Without loss of generality, $x\in \tild e\cup\tild\Gamma_1$. Since $d_{\tild\Gamma}(x,y)\leq \ell (\tild e)$, the path between $x$ and $y$ cannot contain a full edge in the orbit of $\tild e$. This means that $x,y\in\tild e\cup\tild e'\cup \tild\Gamma_1$ for some lift $\tild e'$ of $e$ whose endpoint is in $\tild \Gamma_1$. By \cref{lem: expanding by two copies of minimal edge}, there exists $C'$ (that depend only on $K,C$) such that $$\frac{1}{K}d_{\tild\Gamma}(x,y)-C' \le d_X(\iota(x),\iota(y))\le Kd_{\tild\Gamma}(x,y)+C'.$$ 
\end{proof}
\begin{theorem} (\cite[Ch.3 Th\'eor\`eme 1.4]{coornaert1990geometrie})\label{thm: loc-to-glob}
    For all $K,C$ there exist $L,K',C'$ such that any $L$-local $(K,C)$-quasi geodesic is a $(K',C')$-quasi geodesic.
\end{theorem}

\begin{proof}[Proof of \cref{prop: constants exist}]
By \cref{lem: full map is local qi}, there exist $C'_0$ for which $\iota$ is an $\ell(\tild e)$-local $(K,C'_0)$-quasi-isometric-embedding. 
By \cref{thm: loc-to-glob}, there exist $L,K',C'$ such that every $L$-local $(K,C'_0)$-quasi-geodesic is a $(K',C')$-quasi-geodesic. Therefore, if $\ell(\tild e)\geq L$, then $\iota$ maps geodesics in $\tild\Gamma$ to $(K',C')$-quasi-geodesic. In other words, it is a $(K',C')$-quasi-isometric embedding.
\end{proof}

\begin{lemma}\label{lem: adding-baspoint}
    For all $K,C$ there exist $(K',C')$ such that if $H$ has a $(K,C)$-good core, then $H$ is $(K',C')$-arboreal (i.e. it has a $(K',C')$-good based core).
\end{lemma}
\begin{proof}
    We start with a subgroup $H$ and a $(K,C)$-good $H$-core $\frC_0=(\Gamma_0,\rho_0,\iota_0)$. We want to find a based $H$-core $\frC_*=(\Gamma,\tild*,\rho,\iota)$ which is $(K',C')$-good, for $(K',C')$ that do not depend on $H$ or $\frC_0$. If $1\in\iota(\tild\Gamma_0)$, set $\tild*=\iota^{-1}(1)$, and if needed, make $*=\pi_{\Gamma_0}(\tild*)$ a vertex.
    
    Now, suppose $1\notin\iota(\tild\Gamma_0)$. Let $\frC_*=(\Gamma,\tild *,\iota,\rho)$ be the based core obtained from $\frC_0$ by attaching a leaf: $\Gamma=\Gamma_0\cup\{*\}\cup e$, where a chosen $\tild *\in \tild\Gamma$ is mapped to 1, and $e$ is defined by the fact that a lift $\tild e$ is mapped isometrically to the shortest path between $1$ and $\iota(\tild\Gamma_0)$; The map $\iota$ is then well-defined, and $\rho$ extends naturally from $\rho_0$. The edge $e$ is a minimal edge, and $\Gamma$ decomposes along $e$ into $\Gamma_0$ and the trivial core $\{*\}$. By \cref{prop: constants exist}, there exist $K',C',L$ (that only depend on $(K,C)$ such that if $\ell(e)\geq L$ then $\frC_*$ is $(K',C')$-good. Therefore, if $\ell(e)\geq L$ we are done. If not, then all we did to $\frC_0$ in order to get $\frC_*$ is add very short leaves, and we get that $\frC_*$ is $(K,C+2L)$-good.
\end{proof}

\section{Uniform constants}\label{sec: uniform constants}
We have seen in the previous section that if a core $\frC$ decomposes into good subcores $\frC_1,\frC_2$ and a sufficiently long minimal edge $\tild e$, then $\frC$ is good. This gives us a strategy of proving \cref{thm: quasiconvex is arboreal} - we take a core  $\frC$ that has minimal size (and is therefore unfoldable). If all of the edges are short, then $\frC$ is one of finitely many possible cores. If one of the edges is long, we can look at the adjacent subcores and use \cref{prop: constants exist}. For this to work, we need to make sure that the longest edge is either long enough for \cref{prop: constants exist} or is short enough. This leads us to the following definition:
\begin{definition}\label{def: compatible}
    Let $\calC =(K_n, C_n)_{n=0}^\infty, \calL=(L_n)_{n=1}^\infty$ be two lists of constants. We say $\calC$ and $\calL$ are compatible if for all $n$:
    \begin{enumerate}
        \item Suppose $\frC=(\Gamma,\rho,\iota)$ is an $H$-core with at most $n$ edges ($H$ could be any free quasi-convex subgroup). If the length of every edge of $\Gamma$ is at most $L_n$, then $\frC$ is $(K_n,C_n)$-good.
        \item Suppose $\frC=(\Gamma,\rho,\iota)$ is an $H$-core with at most $n$ edges ($H$ could be any free quasi-convex subgroup). Let $\tild e\in\tild\Gamma$ be a minimal edge, and let $\frC_1,\frC_2$ are the adjacent subcores. Then if $\ell(\tild e)>L_n$, and $\frC_1,\frC_2$ are $(K_{n-1},C_{n-1})$-good, $\frC$ is $(K_n,C_n)$-good.
    \end{enumerate}
\end{definition}
\begin{claim}\label{cl: compatible lists}
    There exists a pair of lists $\calC, \calL$ that are compatible.    
\end{claim}

\begin{proof}
We build $\calC$, $\calL$ by induction. For the base case, we take $K_0=1,C_0=0$ - a core with $0$ edges is just a vertex, and it is mapped isometrically. Now, given $K_n,C_n$ we want to construct $K_{n+1},C_{n+1},L_{n+1}$. 

By \cref{prop: constants exist} there exist $L_{n+1}, K_{n+1}',C_{n+1}'$ such that if 
a core $\frC$ decomposes along a minimal edge $\tild e$ to $(K_n,C_n)$-good subcores $\frC_1,\frC_2$ as in $\S\ref{sec: expanding quasi isometry}$ and $\ell(\tild e)>L_{n+1}$, then $\frC$ is $(K_{n+1}^t,C_{n+1}^t)$-good. 

To satisfy the first condition of \cref{def: compatible} consider all the cores with at most $n$ edges, each shorter than $L_n$. There are only finitely many of them (up to post-composing with an isometry and conjugating the subgroup). This is true since $X$ is localy finite - the image of the tree of such a core has a fundamental domain that lies in a ball of bounded radius. Since there are only finitely many, we can find constants $K_{n+1}'',C_{n+1}''$ such that they are all $(K_{n+1}'',C_{n+1}'')$-good. 

Setting $K_{n+1} = \max\{K_{n+1}',K_{n+1}''\}, C_{n+1} = \max\{ C_{n+1}',C_{n+1}''\}$ we get that $(K_i,C_i)_{i=1}^{n+1},(L_i)_{i=1}^{n+1}$ are compatible up to $n+1$.
\end{proof}
\begin{fact}\label{fact: finitely many graphs}
    Suppose $H$ is a free group. Up to homeomorphism, there are only finitely many graphs $\Gamma$ such that:
    \begin{enumerate}
        \item $\Gamma$ is finite.
        \item $\Gamma$ has no leaves.
        \item $\pi_1(\Gamma)\simeq H$.
    \end{enumerate}
\end{fact}
\begin{corollary}\label{cor: bounded number of edges}
    For all $r$ there exists $n(r)$ such that: If $\Gamma$ is a finite coarse graph with no leaves, and $\pi_1(\Gamma)\simeq F_r$, then $\Gamma$ has at most $n(r)$ edges.
\end{corollary}
\begin{lemma}\label{lem: uniform-core-no-base}
    For all $r\in\bbN$ there exist $K,C$ such that every quasiconvex free subgroup of rank $r$ has a $(K,C)$-good core.
\end{lemma}
\begin{proof}
     By \cref{cl: compatible lists}, there exists a pair of compatible lists $\calC=(K_n,C_n),\calL=(L_n)$. We will prove that every free quasiconvex subgroup of rank $r$ is $(K_{n(r)},C_{n(r)})$-arboreal, for $n(r)$ from \cref{cor: bounded number of edges}. Suppose $H$ is a quasiconvex free subgroup of rank $r$. let $\frC_H=(\Gamma_H,\rho_H,\iota_H)$ be a minimal $H$-core (i.e. $\sigma(\frC_H)$ is minimal among all $H$-cores). We may assume $\Gamma_H$ is a coarse graph, and so by \cref{cor: bounded number of edges} it has at most $n$ edges. We claim that $\frC_H$ is $({K}_{n(r)},{C}_{n(r)})$-good. Let $e\in \Gamma_H$ be an edge of maximal length (among edges of $\Gamma_H$). Since $\Gamma_H$ is minimal, $e$ is minimal (if there was an improvement of $\frC_H$ at $e$, in particular there would have been an $H$-core that is smaller then $\frC_H$). If $\ell(e)\le L_{n(r)}$ we are done by \cref{def: compatible}. Assume $\ell(e)>L_{n(r)}$. Suppose $\frC_1,\frC_2$ are the subcores adjacent to some lift of $e$. If we prove that $\frC_1,\frC_2$ are both $(K_{n(r)-1},C_{n(r)-1})$-good, we finish by part 2 of \cref{def: compatible}. To do that, we can apply the same idea. Every edge of $\frC_1,\frC_2$ is minimal, since an improvement to one of the subcores is also an improvement to $\frC$. So we can continue the process of removing the longest edge, and looking at the subcores. Since the number of edges reduces, this process must finish, thus proving that $\frC_H$ was indeed $(K_{n_r},C_{n_r})$-good.
\end{proof}
\begin{proof}[Proof of \cref{thm: quasiconvex is arboreal}]
Let $r\in\bbN$. By \cref{lem: uniform-core-no-base}, there exist $K_0,C_0$ such that every free quasiconvex subgroup of rank $r$ has a $(K_0,C_0)$-good core. By \cref{lem: adding-baspoint}, there exist $(K,C)$ such that every subgroup that has a $(K_0,C_0)$-good core is $(K,C)$-arboreal. In particular, every free quasiconvex subgroup of rank $r$ is $(K,C)$-arboreal.
\end{proof}

\section{Maps between cores}\label{sec: maps between cores}
Let $G$ be a hyperbolic group. Suppose that $H_1\leq H_2\leq G$, and suppose that $H_1$ and $H_2$ are $(K,C)$-arboreal, with based cores $\frC_1=(\Gamma_1,\tild*_1,\rho_1,\iota_1),\frC_2=(\Gamma_2,\tild*_2,\rho_2,\iota_2)$.
Assume that every point in $\tild\Gamma_1$ lies on a path between two points in the orbit of $\tild*_1$, i.e. $\text{Hull}(H_1.\tild*_1)=\tild{\Gamma_1}$ (otherwise, replace $\tild{\Gamma_1}$ with the subtree $\text{Hull}(H_1.\tild*_1)$).
We would like to construct a map $\psi:\Gamma_1\to\Gamma_2$.
\begin{observation}\label{lem: embeddings are close}
There exists $D=D(K,C,\delta)$ such that
$\iota_1(\tild\Gamma_1)\subseteq\calN_D(\iota_2(\tild\Gamma_2))$.    
\end{observation}

\begin{proof}
The $\iota_1$-image of every vertex of $\tild{\Gamma_1}$ is also in the image of $\iota_2$. Let $\tild x\in\tild{\Gamma_1}$ be a point. Then $\tild x$ lies in a path between two points in the orbit of $\tild*_1$, which we will denote by $h.\tild*_1,g.\tild*_1$. This means that $\iota_1(\tild x)$ is in a $(K,C)$-quasi-geodesic between $h=\iota_1(h.\tild*_1)$ and $g=\iota_1(h.\tild*_1)$. 
The image $\iota_2(\tild\Gamma_2)$ also has a $(K,C)$-quasi-geodesic between $h$ and $g$. By the Morse lemma, quasi-geodesics are close, which means there exists $D=D(K,C)-\delta$ such that $x\in \calN_D(\iota_2(\tild\Gamma_2))$. 
\end{proof}
\begin{lemma} \label{lem: map between cores}
There exist $K',C'$ that depend only on $K,C,\delta$, and a continuous map 
$\psi:\Gamma_1\to\Gamma_2$ which induces the inclusion of $H_1$ into $H_2$ in $\pi_1$, and whose lift to the universal covers
$\tild\psi:\tild{\Gamma}_1\to \tild\Gamma_2$ is a $(K',C')$-q.i embedding.
\end{lemma} 
\begin{proof}
We will construct the map $\psi$ in two steps:

\textbf{Step 1.}  Let us construct an $H_1$-equivariant $(K'_0,C'_0)$-quasi-isometric embedding $\tild\psi_0: \tild\Gamma_1\to \tild\Gamma_2$.

Take a connected (strict) fundamental domain $\Omega$ for the action of $H_1$ on $\tild\Gamma_1$.
For $x\in\Omega$, define $\tild\psi_0(x)$ to be a point $y$ on $\tild\Gamma_2$ such that $d_{X}(\iota_1(x),\iota_2(y))$ is minimal (this distance is less than $D$ by the result of \cref{lem: embeddings are close}).
Any other point in $\tild\Gamma_1$ is of the (unique) form $g.x$ for some $x\in\Omega$, so define $\tild\psi_0(g.x)=g.\tild\psi_0(x)$. $\tild\psi_0$ is well defined, and $H_1$-equivariant.
All that remains is to find constants and show that $\tild\psi_0$ is a quasi-isometric embedding. Let $x,x'\in \Omega$, denote $y=\tild\psi_0(x), y'=\tild\psi_0(x')$, and let $g,g'\in H_1$. 
We bound $d_{\tild\Gamma_2}(g.y,g'.y')$:
\begin{align*}
    d_{\tild\Gamma_2}(g.y,g'.y')&\leq K\cdot d_{X}(g.\iota_2(y),g'.\iota_2(y'))+KC\\&\leq
    K\cdot (d_{X}(g.\iota_2(y), g.\iota_1(x))+d_{X}(g.\iota_1(x),g'.\iota_1(x'))\\ &\quad+d_{X}(g'.\iota_1(x'),g'.\iota_2(y')))+KC
    \\
    &\leq K\cdot d_{X}(g.\iota_1(x),g'.\iota_1(x'))
    +2KD+KC\\
    &\leq K^2d_{\tild\Gamma_1}(g.x,g'.x')+2KC+2KC\\
    \\
    d_{\tild\Gamma_2}(g.y,g'.y')&\geq K^{-1}\cdot d_{X}(g.\iota_2(y),g'.\iota_2(y'))-K^{-1}C\\
    &\geq K^{-1}\cdot (d_{X}(g.\iota_1(x), g'.\iota_1(x'))-d_{X}(g.\iota_1(x),g.\iota_2(y))\\&\quad-d_{X}(g'.\iota_1(x'),g'.\iota_2(y')))-K^{-1}C\\&\geq K^{-1}\cdot d_{X}(g.\iota_1(x),g'.\iota_1(x'))-2K^{-1}D-K^{-1}C
    \\&\geq K^{-2}d_{\tild\Gamma_1}(g.x,g'.x')-2K^{-1}D-2K^{-1}C
\end{align*}
Therefore, $\tild\psi_0$ is a $(K'_0,C'_0)$-q.i. embedding for $K'_0=K^2, C'_0=2KD+2KC$.

\textbf{Step 2.}  Let us use $\tild\psi_0$ to construct an $H_1$-equivariant $(K',C')$-quasi-isometric embedding $\tild\psi:\tild\Gamma_1\to \tild\Gamma_2$ which is also continuous.

Subdivide $\tild\Gamma_1$, such that the new edges have length which is less than or equal to $1$ (and do it in such that the action remains simplicial, by subdividing edges that are in the same orbit in the same way). Denote the set of vertices (after subdividing) by $V'$, and the set of edges by $E'$. For $v\in V'$, define $\tild\psi(v)=\tild\psi_0(v)$. For an edge $e\in E'$ between vertices $v$ and $v'$, define $\tild\psi$ such that it would map $e$ to the geodesic between $\tild\psi(v)$ and $\tild\psi(v')$ linearly. Each point on $e$ is at distance at most $\frac{1}{2}$ from one of $v$ and $v'$, and each point on $\tild\psi(e)$ is at distance at most $\frac{K+C}{2}$ from one of $\tild\psi(v),\tild\psi(v')$ (since the distance between $\tild\psi(v)$ and $\tild\psi(v')$ is at most $K'_0+C'_0$). Let $x,y\in \tild\Gamma_1$ be such that the vertices closest to them are $v_x, v_y$ respectively. Then:
\begin{align*}
    d_{\tild\Gamma_2}(\tild\psi(x),\tild\psi(y)) &\leq d_{\tild\Gamma_2}(\tild\psi(v_x),\tild\psi(v_y))+2\cdot \frac{K'_0+C'_0}{2} \\
    & = d_{\tild\Gamma_2}(\tild\psi_0(v_x),\tild\psi_0(v_y))+K'_0+C'_0 \\
    & \leq K'_0\cdot d_{\tild\Gamma_1}(v_x,v_y)+K'_0+2C'_0\\
    & \leq K'_0\cdot d_{\tild\Gamma_1}(x,y)+3K'_0+2C'_0\\
    \\
    d_{\tild\Gamma_2}(\tild\psi(x),\tild\psi(y)) & \geq d_{\tild\Gamma_2}(\tild\psi(v_x),\tild\psi(v_y))-2\cdot\frac{K'_0+C'_0}{2}\\
    & = d_{\tild\Gamma_2}(\tild\psi_0(v_x),\tild\psi_0(v_y))-K'_0-C'_0\\
    & \geq {K'_0}^{-1}\cdot d_{\tild\Gamma_1}(v_x,v_y)-K'_0-2C'_0\\
    & \geq {K'_0}^{-1}\cdot d_{\tild\Gamma_1}(x,y)-2{K'_0}^{-1} -K'_0-2C'_0\\
    & \geq {K'_0}^{-1}\cdot d_{\tild\Gamma_1}(x,y) -3K'_0-2C'_0
\end{align*}
So, for $K'=K'_0, C'=3K'_0 +2C'_0$, $\tild\psi$ is a $(K',C')-$q.i. embedding.

The map $\tild\psi$ is continuous, and it is $H_1$-equivariant, since $H_1\actson \tild\Gamma_2$ by isometries ($\tild\psi(g.v)=g.\tild\psi(v)$ for a vertex $v$, and it works for the rest of the points since the action sends geodesics to geodesics). Therefore, there is a well-defined map $\psi:\Gamma_1\to\Gamma_2$ which induces the inclusion $H_1\into H_2$, and lifts to $\tild\psi$.
\end{proof}
\section{Bounding cores}\label{sec: bounding cores}

Let $A$ be a finite subset of $G$. The displacement of $A$ at $1$ is defined by
$$\tau_1(A)=\sum_{a\in A} d_X(1, a).$$

\begin{remark}
    Given a constant $\alpha>0$, the set $\{A\subset G| A \text{ is finite and }\tau_1(A)<\alpha\}$ is finite.
\end{remark}

\begin{lemma} \label{lem: not many small groups}
Let $G$ be a hyperbolic group with a generating set $S$. Then given $\alpha,K,C>0$, there are finitely many $(K,C)$-arboreal subgroups with core $\frC$ such that $\sigma(\frC)\leq\alpha$.
\end{lemma}
\begin{proof}
Let $H\leq G$ be a $(K,C)$-arboreal subgroup of rank $r$ with a based $H$-core $\frC=(\Gamma,\tild*,\rho,\iota).$
We can find a free generating set $A$ for $H=\pi_1(\Gamma,*)$, in the following way:

choose a spanning tree $T$ of $\Gamma$. Let $e$ be an edge outside $T$ between $v_1$ and $v_2$. Denote by $a_e$ the loop obtained by concatenating the path on $T$ between $*=\pi_\Gamma(\tild*)$ and $v_1$, $e$, and the path on $T$ between $v_2$ and $*$.

Now, define $A=\{a_e|e$ is an edge outside of $T\}$

Each loop in $A$ passes through each edge twice at most. Therefore, the length of each loop is less then $2\sigma(\frC)$.

There is a bound on $\tau_1(A)$ (in $G$):
\begin{align*}
    \tau_1(A)&=\sum_{a\in A} d_X(1, a)\\
    &\leq K\cdot(\sum_{a\in A} d_T(\tild*, a.\tild*)) + rC\\
    &= K\cdot(\sum_{a\in A} \ell(a)) + rC\\
    &\leq 2\cdot K\cdot r \cdot\sigma(\frC)+rC
\end{align*}

So, every $(K,C)$-arboreal subgroup with $\sigma(\frC)\leq\alpha$ has a generating set $A$ with $\tau(A)\leq2Kr\alpha +rC$. There are finitely finite sets with $\tau_1\leq2Kr\alpha+rC$, and we are done.
\end{proof}



\begin{proof}[Proof of \cref{thm: main_plus}]
Let $H$ be a quasiconvex free subgroup of $G$, and let $k\in \bbN$. Without loss of generality assume that $\rk(H)\le k$. 
By \cref{thm: quasiconvex is arboreal}, there exist $K,C$ such that any free quasiconvex subgroup $H'$ of rank at most $k$ is $(K,C)$-arboreal.
In particular $H$ has a $(K,C)$-good core $\frC=(\Gamma,\rho,\iota)$.

Now, assume $H'$ is a free quasiconvex subgroup of rank at most $k$ which contains $H$, and let 
$\frC'=(\Gamma',\rho',\iota')$ denote its $(K,C)$-good core. By \cref{lem: map between cores}, there exists constants $K',C'$ (which depend only on $K,C$) and a map $\psi:\Gamma\to \Gamma'$, that in $\pi_1$ induce the inclusion $H\leq H'$, and its lift to the universal cover $\tild\psi$ is a $(K',C')$-quasi-isometric embeddings.

If, in addition, $H$ is not contained in a free factor of $H'$ then the map $\psi$ is surjective.
For an edge $e$ of $\Gamma$, we have $\ell(\psi(e))\leq K'\cdot\ell(e)+C'$. Since the images under $\psi$ of edges in $\Gamma$ cover $\Gamma'$, this implies $\sigma(\frC')\leq K'\cdot\sigma(\frC)+C'$.

We found a bound on $\sigma(\frC')$ which depends only on $K,C$. 
By \cref{lem: not many small groups}, there are only finitely many such subgroups $H'$.
\end{proof}

\section{Ascending chains}\label{sec: ascending chains}
We are now back to discussing ascending chains of subgroups, in an effort to prove \cref{thm: ascending quasiconvex stabilizes}.
In view of \cref{thm: main_plus}, we would like to assume that $H_{i-1}$ is not contained in a proper free factor of $H_{i}$, and that $H_1$ is not contained in a proper free factor of $H_{i}$. Bering and Lazarovich proved these assumptions are possible  in \cite{bering2021ascending}. In this section, we include the proofs for completeness.

\begin{lemma}\label{lem: ascneding not in free factors}
Let $G$ be a group, and let $H_1\leq H_2\leq...\leq G$ be an ascending chain of free subgroups of constant rank. Then there exists a chain $K_1\leq K_2\leq...\leq G$ of free subgroups of constant rank, such that $K_i$ is not contained in a proper free factor of $K_{i+1}$. Furthermore ${H_i}$ stabilizes if and only if ${K_i}$ stabilizes.
\end{lemma}
\begin{proof}
By induction on the rank $r$ of the groups $H_i$. For $r=1$, the groups $H_i$ do not have proper free factor, and $K_i=H_i$ suffices. Assume $r>1$. If there is a finite number of indices $i$ such that $H_i$ is contained in a proper free factor of $H_{i+1}$, then taking the tail of the sequence ${H_i}$ we are done. Assume that the set of indices $I\subseteq \bbN$ for which it happens is infinite. Then, for each $i\in I$, $H_i$ is contained in some free factor $K_i$ of $H_{i+1}$. By Grushko's theorem, $\rk(K_i)<\rk(H_{i+1})=r$. The sequence $\{K_i\}_{i\in I}$ is ascending. After passing to a subsequence, we may assume $K_i$ have constant rank $r'<r$. Clearly $K_i$ stabilizes if and only if $H_i$ stabilizes, and since $r'<r$, we can continue by induction.
\end{proof}
\begin{lemma} \label{lem: first not free factor}
Let $G$ be a group, and let $H_1\leq H_2\leq...\leq G$ be an ascending chain of free subgroups of constant rank, such that for all $i$, $H_i$ is not contained in a proper free factor of $H_{i+1}$. Then for all $i$, $H_1$ is not contained in a proper free factor of $H_i$.
\end{lemma}
\begin{proof}
By induction on $i$. For $i=1,2$ there is nothing to prove.
For $i>2$, suppose $K$ is a proper free factor of $H_i$.
By the Kurosh decomposition, $K\cap H_{i-1}$ is a free factor of $H_{i-1}$. And by the assumption $H_{i-1} \not\le K$, so $K\cap H_{i-1}$ is a proper free factor of $H_{i-1}$. By the induction hypothesis $H_1$ is not contained in a proper free factor of $H_{i-1}$, and so it is not contained in $K\cap H_{i-1}$. It follows that $H_1$ is not contained in $K$.
\end{proof}

\begin{proof}[Proof of \cref{thm: ascending quasiconvex stabilizes}]
    Let $H_1\le H_2 \le \dots$ be an ascending sequence of quasiconvex free subgroups of rank $r$.
By \cref{lem: ascneding not in free factors} we may assume that for all $i$, $H_i$ is not contained in a proper free factor of $H_{i+1}$. By \cref{lem: first not free factor}, $H_1$ is not contained in a proper free factor of $H_i$.
    By \cref{thm: main_plus} there are only finitely many such $H_i$, and so the sequence stabilizes.
\end{proof}

\bibliographystyle{plain}
\bibliography{biblio}

\begin{thebibliography}{10}

\bibitem{arzhantseva1996class}
Gul'nara~Nurullovna Arzhantseva and A~Yu Ol'shanskii.
\newblock The class of groups all of whose subgroups with lesser number of
  generators are free is generic.
\newblock {\em Mathematical Notes}, 59(4):350--355, 1996.

\bibitem{bering2021ascending}
Edgar~A Bering~IV and Nir Lazarovich.
\newblock Ascending chains of free groups in 3-manifold groups.
\newblock {\em arXiv preprint arXiv:2111.11777}, 2021.

\bibitem{bestvina1991bounding}
Mladen Bestvina and Mark Feighn.
\newblock Bounding the complexity of simplicial group actions on trees.
\newblock {\em Inventiones mathematicae}, 103(1):449--469, 1991.

\bibitem{coornaert1990geometrie}
Michel Coornaert, Thomas Delzant, and Athanase Papadopoulos.
\newblock G{\'e}om{\'e}trie et th{\'e}orie des groupes.
\newblock {\em Lecture Notes in Mathematics}, 1990.

\bibitem{dunwoody1998folding}
MJ~Dunwoody.
\newblock Folding sequences.
\newblock {\em Geometry and Topology monographs}, 1:143--162, 1998.

\bibitem{gromov1987hyperbolic}
Mikhael Gromov.
\newblock Hyperbolic groups.
\newblock In {\em Essays in group theory}, pages 75--263. Springer, 1987.

\bibitem{higman1951finitely}
Graham Higman.
\newblock A finitely related group with an isomorphic proper factor group.
\newblock {\em Journal of the London Mathematical Society}, 1(1):59--61, 1951.

\bibitem{kapovich2023ascending}
Ilya Kapovich.
\newblock Ascending chain condition in generic groups.
\newblock {\em arXiv preprint arXiv:2306.06706}, 2023.

\bibitem{kapovich2002stallings}
Ilya Kapovich and Alexei Myasnikov.
\newblock Stallings foldings and the subgroup structure of free groups.
\newblock {\em arXiv preprint math/0202285}, 2002.

\bibitem{kapovich2004freely}
Ilya Kapovich and Richard Weidmann.
\newblock Freely indecomposable groups acting on hyperbolic spaces.
\newblock {\em International Journal of Algebra and Computation},
  14(02):115--171, 2004.

\bibitem{kapovich2005foldings}
Ilya Kapovich, Richard Weidmann, and Alexei Myasnikov.
\newblock Foldings, graphs of groups and the membership problem.
\newblock {\em International Journal of Algebra and Computation},
  15(01):95--128, 2005.

\bibitem{shusterman2017ascending}
Mark Shusterman.
\newblock Ascending chains of finitely generated subgroups.
\newblock {\em Journal of Algebra}, 471:240--250, 2017.

\bibitem{stallings1983topology}
John~R Stallings.
\newblock Topology of finite graphs.
\newblock {\em Inventiones mathematicae}, 71(3):551--565, 1983.

\bibitem{takahasi1951note}
Mutuo Takahasi.
\newblock Note on chain conditions in free groups.
\newblock {\em Osaka Mathematical Journal}, 3(2):221--225, 1951.

\bibitem{weidmann2002nielsen}
Richard Weidmann.
\newblock The nielsen method for groups acting on trees.
\newblock {\em Proceedings of the London Mathematical Society}, 85(1):93--118,
  2002.

\bibitem{weidmann2024foldings}
Richard Weidmann and Thomas Weller.
\newblock Foldings in relatively hyperbolic groups.
\newblock {\em arXiv preprint arXiv:2403.17686}, 2024.

\end{thebibliography}
\end{document}